\documentclass[a4paper]{amsart}
\usepackage{graphicx}
\usepackage{graphics}
\usepackage{amsmath}
\usepackage{latexsym}
\usepackage{amssymb}
\usepackage{subfigure}
\usepackage[all]{xy}
\xyoption{matrix}
\xyoption{arrow}
\usepackage{psfrag}

\DeclareMathOperator{\ind}{ind}
\DeclareMathOperator{\Hom}{Hom}
\DeclareMathOperator{\Ext}{Ext}
\DeclareMathOperator{\PP}{\mathcal{P}}
\DeclareMathOperator{\module}{mod}

\begin{document}

\renewcommand{\th}{\operatorname{th}\nolimits}
\newtheorem{lemma}{Lemma}[section]
\newtheorem{prop}[lemma]{Proposition}
\newtheorem{corollary}[lemma]{Corollary}
\newtheorem{theorem}[lemma]{Theorem}
\newtheorem{remark}[lemma]{Remark}
\newtheorem{definition}[lemma]{Definition}
\newtheorem{example}[lemma]{Example}

\title[Categorification of a frieze pattern determinant]{Categorification of a frieze pattern determinant}

\author[Baur]{Karin Baur}
\address{
Institut f\"{u}r Mathematik und wissenschaftliches Rechnen \\
Universit\"{a}t Graz \\ Heinrichstrasse 36 \\ A-8010 Graz \\ Austria}
\email{baurk@uni-graz.at}

\author[Marsh]{Robert J. Marsh}
\address{School of Mathematics \\
University of Leeds \\
Leeds LS2 9JT \\
England
}
\email{marsh@maths.leeds.ac.uk}


\begin{abstract}
Broline, Crowe and Isaacs have computed the determinant of a matrix
associated to a Conway-Coxeter frieze pattern. We generalise their result
to the corresponding frieze pattern of cluster variables arising from the
Fomin-Zelevinsky cluster algebra of type $A$. We give a
representation-theoretic interpretation of this result in terms of certain
configurations of indecomposable objects in the root category of type $A$.
\end{abstract}

\thanks{This work was supported by the
Engineering and Physical Sciences Research Council
[grant numbers EP/S35387/01 and EP/G007497/1] and the
Forschungsinstitut f\"{u}r Mathematik (FIM) at the ETH, Z\"{u}rich.
Karin Baur was supported by the Swiss National Science Foundation (grant
PP0022-114794).}

\date{3 February 2012}

\maketitle

\section{Introduction} \label{s:introduction}

Consider a generic $2\times n$ matrix, i.e.\ a matrix
$$X=
\begin{pmatrix}
x_1 & x_2 & \cdots & x_n \\
y_1 & y_2 & \cdots & y_n
\end{pmatrix}
$$
whose entries are indeterminates. For a choice of two columns of $X$,
$1\leq i,j\leq n$, let
$$\Delta_{ij}=\begin{vmatrix} x_i & x_j \\ y_i & y_j \end{vmatrix}$$
be the corresponding minor of $X$ (so $\Delta_{ii}=0$).
Let $A$ be the symmetric matrix with entries $A_{ij}$ given by:
$$A_{ij}=\begin{cases}
\Delta_{ij} & \text{if $i\leq j$;} \\
\Delta_{ji} & \text{if $i>j$.}
\end{cases}$$

Our main result is the following.

\begin{theorem} \label{t:introdetresult}
$$\det(A)=(-2)^{n-2}\Delta_{12}\Delta_{23}\cdots \Delta_{n-1,n}\Delta_{n1}.$$
\end{theorem}

Our motivation comes from a result of Broline, Crowe and Isaacs~\cite{bci}
concerning \emph{frieze patterns} of integers.
Theorem~\ref{t:introdetresult} can be regarded as a generalisation of this result, which we now describe.

Frieze patterns of integers in the plane were considered
in~\cite{coxeter,conwaycoxeter1,conwaycoxeter2} by Conway and Coxeter.
Such a frieze pattern consists of a finite number of infinite rows of
integers, with each row interlacing its neighbouring rows, and satisfies the
{\em unimodular rule},
which states that for every four adjacent numbers forming a square:
$$\begin{matrix}
& b & \\
a & & d \\
& c &
\end{matrix}$$
the relation $ad-bc=1$ is satisfied. The entries in the first and last row
are zero; the entries in the second and penultimate rows are $1$,
and all other entries should be positive.

Fix an integer $n\geq 3$.
A frieze pattern is said to be of {\em order} $n$ if it has $n+1$ rows.
We fix a regular $n$-sided polygon $\PP_n$, with vertices $1,2,\ldots ,n$
numbered in cyclic order, arranged clockwise around the boundary
(we work with the vertices modulo $n$, with representatives in
$\{1,2,\ldots ,n\}$).
In~\cite{conwaycoxeter1,conwaycoxeter2} it is shown that a frieze pattern
can be obtained from a triangulation $\pi$ of $\PP_n$ in the following way.
For each pair of integers $i,j\in \{1,2,\ldots ,n\}$, define an integer $m_{ij}$
as follows.
Set $m_{ii}=0$ and $m_{i,i+1}=1$ for all $i$.
Then let $m_{i-1,i+1}$ be the number of triangles in $\pi$ incident with
vertex $i$. Define $m_{ij}$ for all $i<j$ inductively using the formula
\begin{equation}
\label{e:mijformula}
m_{i-1,j+1}=\frac{m_{i-1,j}m_{i,j+1}-1}{m_{ij}}.
\end{equation}
Then set $m_{ji}=m_{ij}$ for all $i<j$. The numbers $m_{ij}$ are all
positive integers and, when arranged as in Figure~\ref{f:friezepattern},
form a frieze pattern. Furthermore, every frieze pattern of order $n$
arises from a triangulation of $\PP_n$ in this way.

\begin{figure}[htp]
\psfragscanon
\psfrag{0}{$0$}
\psfrag{m12}{$m_{12}$}
\psfrag{m13}{$m_{13}$}
\psfrag{m23}{$m_{23}$}
\psfrag{m1n-1}{$m_{1,n-1}$}
\psfrag{m1n}{$m_{1n}$}
\psfrag{m2n}{$m_{2n}$}
\psfrag{mn-2n-1}{$m_{n-2,n-1}$}
\psfrag{mn-2n}{$m_{n-2,n}$}
\psfrag{mn-1n}{$m_{n-1,n}$}
\includegraphics[width=12cm]{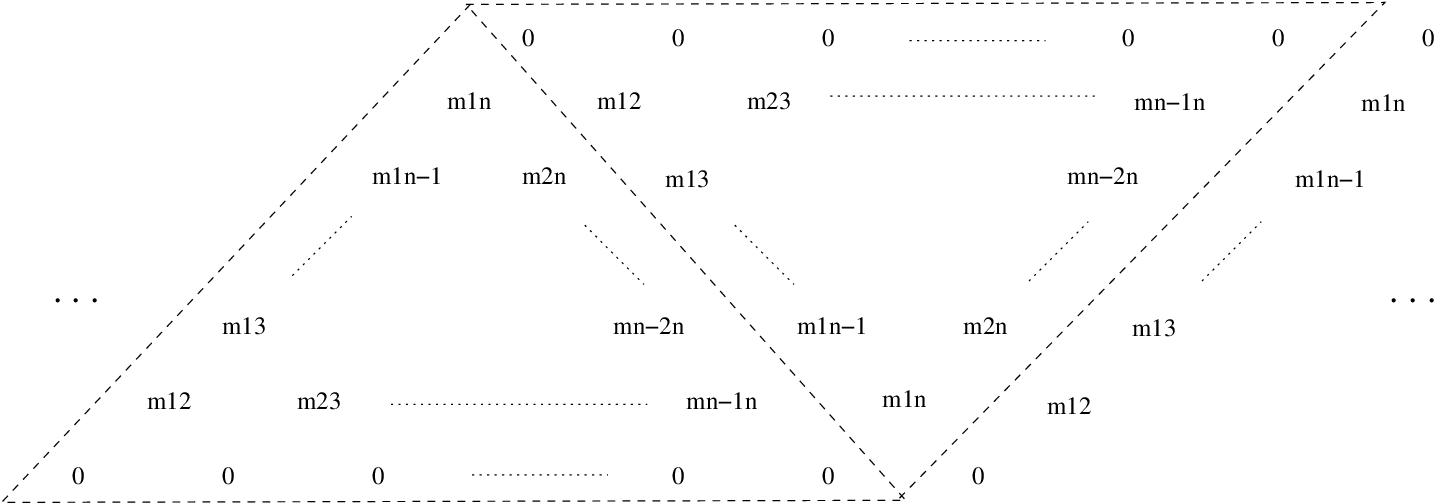}
\caption{Rule for writing out a frieze pattern}
\label{f:friezepattern}
\end{figure}

Note that the entries lying in a triangle below $m_{1,n}$ form a
fundamental domain for a glide reflection preserving the pattern.
This fundamental domain is indicated by a dashed line triangle (together
with its shift to the right).

For example, the frieze pattern corresponding to the triangulation in
Figure~\ref{f:exampletriangulation} is shown in Figure~\ref{f:friezeexample}.
The middle triangle indicates the fundamental domain mentioned above;
its images under the glide reflection and its inverse are also displayed.

\begin{figure}
\includegraphics[width=4cm]{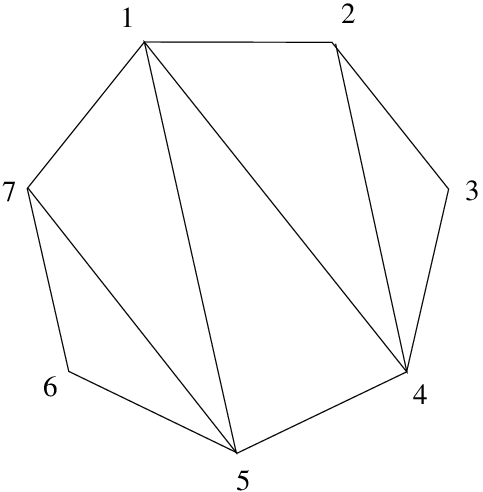}
\caption{A triangulation of $\PP_7$}
\label{f:exampletriangulation}
\end{figure}

\begin{figure}[htp]
\psfragscanon
\psfrag{1}{$1$}
\psfrag{2}{$2$}
\psfrag{3}{$3$}
\psfrag{4}{$4$}
\psfrag{5}{$5$}
\psfrag{6}{$6$}
\psfrag{7}{$7$}
\psfrag{8}{$8$}
\psfrag{9}{$9$}
\includegraphics[width=12cm]{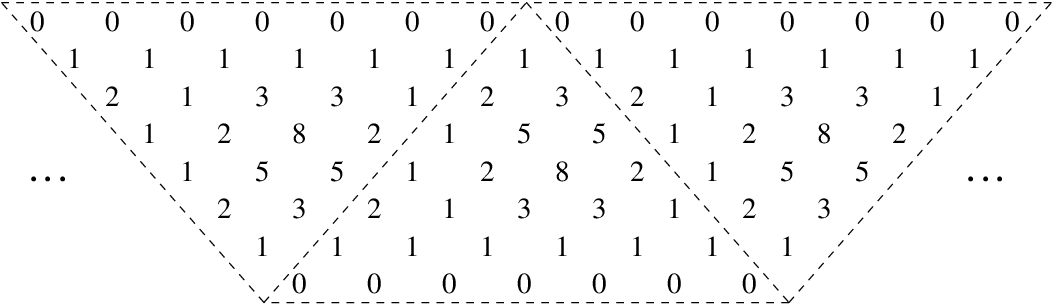}
\caption{The frieze pattern corresponding to the triangulation in Figure~\ref{f:exampletriangulation}.}
\label{f:friezeexample}
\end{figure}

The result of Broline, Crowe and Isaacs can be stated as follows.

\begin{theorem} \cite{bci} \label{t:bciresult}
Let $n\geq 3$, let $\pi$ be a triangulation of $\PP_n$ and let
$M(\pi)=(m_{ij})$ be the symmetric matrix defined above.
Then $\det(M)=-(-2)^{n-2}$.
\end{theorem}

For the example in Figure~\ref{f:exampletriangulation}, we have
$$M(\pi)=
\begin{pmatrix}
0&1&2&1&1&2&1 \\
1&0&1&1&2&5&3 \\
2&1&0&1&3&8&5 \\
1&1&1&0&1&3&2 \\
1&2&3&1&0&1&1 \\
2&5&8&3&1&0&1 \\
1&3&5&2&1&1&0
\end{pmatrix}
$$
which has determinant $-(-2)^5=32$.
In addition, a geometric interpretation of all of the entries in the
frieze pattern in terms of $\pi$ is given in~\cite{bci}.

We note that frieze patterns of integers of various kinds have been studied
recently; see, for
example~\cite{ars,baurmarsh2,bergeronreutenauer,guo,kellerscherotzke,moriergenoud,mot,propp}.

A connection between cluster algebras and frieze patterns was established
in the article~\cite{calderochapoton},
which showed that the frieze patterns above can be obtained from cluster
algebras of type $A$.

The homogeneous coordinate ring of the Grassmannian of $2$-planes in a
$n$-dimensional vector space is a cluster algebra of type
$A_{n-3}$~\cite[12.6]{fominzelevinsky2} (see also~\cite[\S1]{fominzelevinsky1}).
The cluster variables are in bijection with the diagonals of $\PP_n$.
If $i,j$ are the end-points of such a diagonal, we write $u_{ij}$ for
the corresponding cluster variable (so $u_{ij}=u_{ji}$).
It has stable variables $u_{ij}$ where $i,j$ are the end-points of a boundary
edge. We set $u_{ii}=0$ for all $i$ and $u_{ij}=u_{ji}$ for $i>j$.

Via the above bijection, the clusters are in bijection with the triangulations
of $\PP_n$. We fix such a triangulation $\pi$ and corresponding cluster.
By the Laurent Phenomenon~\cite[3.1]{fominzelevinsky1},
each cluster variable can be written as a Laurent polynomial in the
elements of the cluster with the coefficients of the polynomial given by
polynomials in the stable variables.

When the variables in the cluster and the stable variables are
specialised to $1$, the resulting integers, when arranged correctly,
produce the corresponding frieze pattern.

Theorem~\ref{t:introdetresult} can be reinterpreted in terms of this cluster
algebra as follows. In this way we see that it is in fact a generalisation of
Theorem~\ref{t:bciresult}.

\begin{theorem} \label{t:introclusterinterpretation}
Let $\pi$ be a triangulation of $\PP_n$. Let $U(\pi)=(u_{ij})$
Then
$$\det(U(\pi))=-(-2)^{n-2}u_{12}u_{23}\cdots u_{n-1,n}u_{n1}.$$
\end{theorem}

We go on to show that this result can be given a categorical
interpretation in terms of the root category of type $A_{n-1}$. By interpreting
this category as a category of oriented edges between vertices of $\PP_n$ (using
methods similar to those in~\cite{ccs}) we show that the above determinant
can be reinterpreted as a sum over configurations of indecomposable
objects in the root category. Each configuration is a maximal collection
of indecomposable objects such that no object lies in the frame of the
other (see Section~\ref{s:frames} for the definition of frame) which is also
of maximal cardinality.

\section{A determinantal result}
In this section, we prove our main result:
\begin{theorem} \label{t:detresult}
Let $A=(A_{ij})$ be the matrix of minors of $X$ defined above.
Then
$$\det(A)=(-2)^{n-2}\Delta_{12}\Delta_{23}\cdots \Delta_{n-1,n}\Delta_{n1}.$$
\end{theorem}
\begin{proof}
The key point is that the minors $\Delta_{ij}$ satisfy the Pl\"{u}cker
relations, i.e., whenever $1\leq p<q<r<s\leq n$ are vertices of $\PP_n$, we have:
\begin{equation}
\Delta_{pq}\Delta_{rs}+\Delta_{qr}\Delta_{ps}=\Delta_{pr}\Delta_{qs}.
\end{equation}
Noting that $A$ is symmetric, it follows that, whenever $i,j,k,l$ are arranged
clockwise around the boundary of $\PP_n$, we have:
\begin{equation} \label{e:Arelations}
A_{ij}A_{kl}+A_{jk}A_{il}=A_{ik}A_{jl}.
\end{equation}

We use induction on $n$, showing that the Pl\"{u}cker relations are sufficient
to imply the result.
For $n=3$ we have:
$$A=
\begin{pmatrix}
0 & \Delta_{12} & \Delta_{13} \\
\Delta_{12} & 0 & \Delta_{23} \\
\Delta_{13} & \Delta_{23} & 0
\end{pmatrix}
$$
which has determinant $-2\Delta_{12}\Delta_{23}\Delta_{31}$ as required.
Now suppose that $n\geq 4$ and that the result is true for $n-1$.
Fix vertices $a,b$ of $\PP_n$ such that $b$ is distinct from
$a-1,a,a+1$. From the quadrilateral in Figure~\ref{specialquad}
we have the Pl\"{u}cker relation:
$$A_{a-1,a}A_{a+1,b}+A_{a,a+1}A_{a-1,b}=A_{a-1,a+1}A_{ab}.$$

\begin{figure}
\begin{center}
\includegraphics{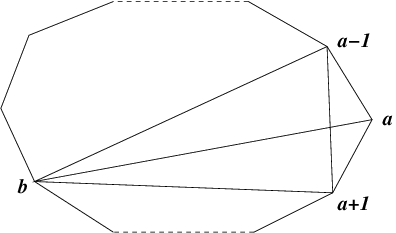} 
\caption{A quadrilateral in $\PP_n$}
\label{specialquad}
\end{center}
\end{figure}

We note that $A_{a-1,a+1}\not=0$. Thus it follows from the above that:
$$\frac{A_{a-1,a}}{A_{a-1,a+1}}A_{a+1,b}+\frac{A_{a,a+1}}{A_{a-1,a+1}}A_{a-1,b}=A_{ab}.$$
Let $R_a$ denote the $a$th row of $A$.
From the above, it follows that if we replace $R_a$ with the row:
$$R_a-\frac{A_{a-1,a}}{A_{a-1,a+1}}R_{a+1}-\frac{A_{a,a+1}}{A_{a-1,a+1}}R_{a-1},$$
then the $a,b$ entry $A'_{ab}$ in the new matrix $A'$ will be $0$ for any
$b\neq a-1,a$ or $a+1$. Noting that $A_{a-1,a-1}=0$, we also have:
$$A'_{a,a-1}=A_{a,a-1}-\frac{A_{a-1,a}}{A_{a-1,a+1}}A_{a+1,a-1}-
\frac{A_{a,a+1}}{A_{a-1,a+1}}A_{a-1,a-1}=0,$$
and, similarly, $A'_{a,a+1}=0$, while
$$A'_{aa}=A_{aa}-\frac{A_{a-1,a}}{A_{a-1,a+1}}A_{a+1,a}-\frac{A_{a,a+1}}{A_{a-1,a+1}}A_{a-1,a}=\frac{-2A_{a-1,a}A_{a,a+1}}{A_{a-1,a+1}}.$$
Expanding along the $a$th row, we obtain that the determinant of $A$ is
\begin{equation}\label{indstep}
\det(A)=(-1)^{2a}\frac{(-2)\Delta_{a-1,a}\Delta_{a,a+1}}{\Delta_{a-1,a+1}}
\det(A(a)),
\end{equation}
where $A(a)$ denotes the matrix obtained by removing the $a$th row and the
$a$th column from $A$. Note that the entries of $A(a)$ satisfy the
relations~\eqref{e:Arelations} for the polygon $\PP_{n-1}$ with vertices
parametrised clockwise using the numbers $\{1,2,\ldots ,n\}\setminus \{a\}$.
By the induction hypothesis,
$$\det(A(a))=(-2)^{n-3}\Delta_{a-1,a+1}\Delta_{12}\Delta_{23}\cdots \hat{\Delta}_{a-1,a}\hat{\Delta}_{a,a+1}\cdots \Delta_{n-1,n}\Delta_{n1},$$
where the hats indicate omission.
Combining this with equation \eqref{indstep} we obtain that
$$\det(A)=(-2)^{n-2}\Delta_{12}\Delta_{23}\cdots \Delta_{n-1,n}\Delta_{n1},$$
as required.
\end{proof}

Note that the identity in Theorem~\ref{t:detresult} can be restated as:
\begin{equation} \label{e:rewritten}
\det(A)=-(-2)^{n-2}A_{12}A_{23}\cdots A_{n-1,n}A_{n1}
\end{equation}
since $A_{n1}=\Delta_{1n}=-\Delta_{n1}$.

\section{Frieze patterns and cluster algebras}

We consider a cluster algebra $\mathcal{A}$ of type $A_{n-3}$
associated to $\PP_n$, defined over the complex numbers.
This cluster algebra appears
in~\cite[\S1]{fominzelevinsky1}
and is described in detail in~\cite[12.2]{fominzelevinsky3}; see
alternatively~\cite[3.2]{propp} (using a perfect matching model for frieze
patterns due to Gabriel Carroll and Gregory Price; see [loc. cit.] for details).

Let $\pi$ be a triangulation of $\PP_n$ and let $\mathbb{F}$ be the field
of rational functions in the variables $u_{ij}$ where $i,j$ are the end-points
of a diagonal in $\pi$ or a boundary edge of $\PP_n$ (regarding $u_{ij}$ and $u_{ji}$
as equal). Define elements $u_{ij}\in \mathbb{F}$, for $i,j$ the end-points of an arbitrary
diagonal of $\PP_n$, inductively as follows (again with $u_{ij}=u_{ji}$). If $i,j,k,l$
are vertices of $\PP_n$, arranged clockwise around the boundary, and $u_{ij},u_{jk},u_{kl},u_{li}$
and $u_{ik}$ are defined but $u_{jl}$ is not, define $u_{jl}$ by the following
\emph{exchange relation}:
\begin{equation} \label{e:exchange}
u_{ij}u_{kl}+u_{jk}u_{il}=u_{ik}u_{jl}.
\end{equation}
It turns out that the elements $u_{ij}$ are well-defined. The cluster algebra $\mathcal{A}$
is the $\mathbb{C}$-subalgebra of $\mathbb{F}$ generated by the $u_{ij}$ for
$i,j$ the end-points of any diagonal or boundary edge of $\PP_n$.
The generators corresponding to diagonals are known as \emph{cluster
variables} and those corresponding to boundary edges are known as
\emph{stable variables}.
For $n=3$ there are no cluster variables, only stable variables.
The exchange relation~\eqref{e:exchange} holds for any choice of
$i,j,k,l$ arranged clockwise around the boundary of $\PP_n$.

The diagonals in a triangulation of $\PP_n$ determine a corresponding set
of cluster variables of $\mathcal{A}$ known as a \emph{cluster}. The
cluster corresponding to $\pi$ is the \emph{initial cluster}.
By~\cite[3.1]{fominzelevinsky1}, the cluster variables can be
written as Laurent polynomials in any fixed cluster, with coefficients given
by polynomials in the stable variables.

The cluster algebra $\mathcal{A}$ is independent (up to isomorphism)
of the choice of $\pi$. In fact, it is isomorphic to the homogeneous coordinate
ring $\mathbb{C}[Gr_2(n)]$ of the Grassmannian of $2$-dimensional subspaces
of an $n$-dimensional vector space; see~\cite[12.6]{fominzelevinsky2}.
Such a subspace can be described by a $2\times n$ matrix, with rows given by a
choice of spanning vectors, and the $2\times 2$ minors of the matrix give
homogeneous coordinates on the Grassmannian.

The homogeneous coordinate ring is generated by these minors, subject to
the Pl\"{u}cker relations, and under the isomorphism, $u_{ij}$, for $i<j$,
maps to the minor associated to columns $i$ and $j$ of the matrix.
The exchange relations above map to the Pl\"{u}cker relations.

Fix a triangulation $\pi$ of $\PP_n$.
Setting $u_{ii}=0$ for all $i$, we consider the symmetric matrix
$U(\pi)=(u_{ij})$, where the $u_{ij}$ are regarded as Laurent polynomials in the
$u_{ij}$ for $i,j$ the end-points of an edge in $\pi$
with coefficients given by polynomials in the stable variables.

We note the following:
\begin{prop} \cite[5.2]{calderochapoton} (see also~\cite[\S2,\S3]{propp})
\label{p:specialisation}
If the $u_{ij}$, for $i,j$ the end-points of a diagonal
in $\pi$, and the stable variables are all specialised to $1$,
then $U(\pi)$ becomes the matrix $M(\pi)$ defined above.
\end{prop}

Caldero-Chapoton prove this by first showing that, for each $i$,
$u_{i-1,i+1}$ specialises to the number of triangles in $\pi$ incident with
vertex $i$, which coincides with $m_{i-1,i+1}$ in the above definition.
The result then follows from a comparison of formulas~\eqref{e:mijformula}
and~\eqref{e:exchange}.

We can now restate Theorem~\ref{t:detresult} in this context. The proof of
this version is the same, using the exchange relations~\eqref{e:exchange}
(and noting equation~\eqref{e:rewritten}).

\begin{theorem} \label{t:clusterinterpretation}
Let $\pi$ be a triangulation of $\PP_n$. Let $U(\pi)=(u_{ij})$ be the
matrix defined above. Then
$$\det(U(\pi))=-(-2)^{n-2}u_{12}u_{23}\cdots u_{n-1,n}u_{n1}.$$
\end{theorem}

We have thus interpreted Theorem~\ref{t:detresult} as a generalisation of
Theorem~\ref{t:bciresult}.

\section{An Example}
Here we give an example of the result in the previous section.
Let $\pi$ be the triangulation of a pentagon shown in Figure~\ref{pentagon}.
\begin{figure}
\begin{center}
\includegraphics{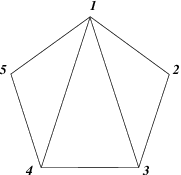} 
\caption{A triangulation of $\PP_5$.}
\label{pentagon}
\end{center}
\end{figure}
Then the corresponding matrix is given by:
$$U(\pi)=
\begin{pmatrix}
0 & u_{12} & u_{13} & u_{14} & u_{15} \\
u_{12} & 0 & u_{23} & \frac{u_{12}u_{34}+u_{23}u_{14}}{u_{13}} &
v \\u_{13} & u_{23} & 0 & u_{34} & \frac{u_{13}u_{45}+u_{15}u_{34}}{u_{14}} \\
u_{14} & \frac{u_{12}u_{34}+u_{23}u_{14}}{u_{13}} &
u_{34} & 0 & u_{45} \\
u_{15} & v & \frac{u_{13}u_{45}+u_{15}u_{34}}{u_{14}} & u_{45} & 0
\end{pmatrix}
$$
where
$$v=\frac{u_{12}u_{13}u_{45}+u_{12}u_{15}u_{34}+u_{14}u_{15}u_{23}}{u_{13}u_{14}}.$$

By Theorem~\ref{t:clusterinterpretation}, we have that
$$\det(u)=8u_{12}u_{23}u_{34}u_{45}u_{51}.$$

\section{A geometric model of the root category}

Let $n\geq 3$ be an integer and let $Q$ be a quiver of type $A_{n-1}$. Let $k$ be
an algebraically closed field. Let $D^b(kQ)$ denote the bounded derived
category of modules over $kQ$, with shift functor $[1]$. Let
$\mathcal{R}_n=D^b(kQ)/[2]$ denote the quotient of $D^b(kQ)$ by the square $[2]$
of the shift. In this section we shall exhibit a geometric
construction of this category (along the lines of~\cite{ccs}).
We remark that this category is sometimes referred to as the
\emph{root category} (of type $A$) since its objects can be put into one-to-one
correspondence with the roots in the corresponding root system (by Gabriel's
Theorem). It was considered in~\cite{happel1}.

We now consider \emph{oriented} edges between vertices of $\PP_n$, denoting the
edge oriented from $i$ to $j$ by $[i,j]$, for any $1\leq i,j\leq n$ with
$i\neq j$ (thus boundary edges are included).

Recall that a \emph{stable translation quiver} is a pair
$(\Gamma,\tau)$ where $\Gamma$ is a locally finite quiver and
$\tau:\Gamma_0\rightarrow \Gamma_0$
is a bijection such that for any $X,Y\in \Gamma_0$, the
number of arrows from $X$ to $Y$ is the same as the number of arrows from
$\tau(Y)$ to $X$.

Let $\Gamma=\Gamma(n)$ be the quiver defined as follows. The set of
vertices, $\Gamma_0$, is the set of all possible oriented edges between
distinct vertices of $\PP_n$ as above.
The arrows, $\Gamma_1$, are of the form $[i,j]\rightarrow [i,j+1]$ and
$[i,j]\rightarrow [i+1,j]$ (where $j+1$ is interpreted as $1$ if $j=n+1$
and similarly for $i+1$), whenever
$[i,j]$ and $[i,j+1]$ (respectively, $[i,j]$ and $[i+1,j]$)
are vertices of $\Gamma$.
Thus an arrow comes from rotating an oriented edge
clockwise about one of its end-points so that the other end-point moves to
an adjacent vertex on the boundary of $\PP_n$.

Let $\tau$ be the automorphism of $\Gamma$ obtained by rotating $\PP_n$
through $2\pi/n$ anticlockwise; thus $\tau([i,j])=[i-1,j-1]$.

\begin{lemma}
The pair $(\Gamma,\tau)$ is a stable translation quiver.
\end{lemma}

The proof is as in~\cite[2.2]{baurmarsh1}: note that this proof also works
in the oriented case we have here.

\begin{example}
We consider the case when $n=5$, so $\PP_n$ is a pentagon. The translation
quiver $(\Gamma(5),\tau)$ is given in Figure~\ref{f:gamma4}.

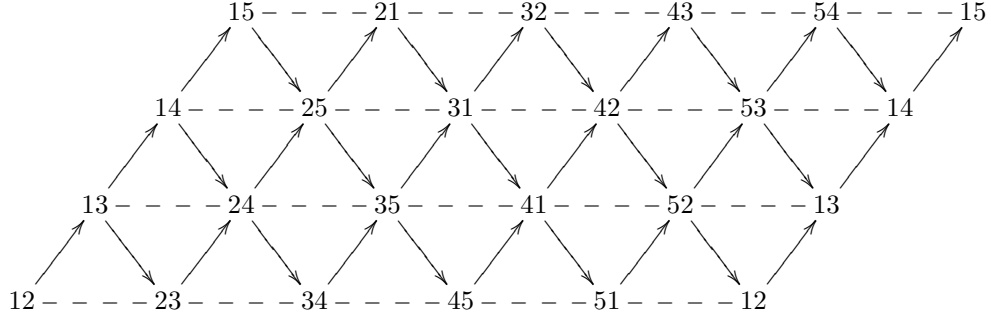
\begin{figure}
$$\xymatrix@C=0.4cm{
&&& 15 \ar@{--}[rr] \ar[dr] && 21 \ar@{--}[rr] \ar[dr] && 32 \ar@{--}[rr] \ar[dr] && 43 \ar@{--}[rr] \ar[dr] && 54 \ar@{--}[rr] \ar[dr] && 15 \\
&& 14 \ar@{--}[rr] \ar[ur] \ar[dr] && 25 \ar@{--}[rr] \ar[ur] \ar[dr] && 31 \ar@{--}[rr] \ar[ur] \ar[dr] && 42 \ar@{--}[rr] \ar[dr] \ar[ur] && 53 \ar@{--}[rr] \ar[ur] \ar[dr]&& 14 \ar[ur] \\
& 13 \ar@{--}[rr] \ar[ur] \ar[dr] && 24 \ar@{--}[rr] \ar[ur] \ar[dr] && 35 \ar@{--}[rr] \ar[ur] \ar[dr] && 41 \ar@{--}[rr] \ar[ur] \ar[dr] && 52 \ar@{--}[rr] \ar[ur] \ar[dr] && 13 \ar[ur] \\
12 \ar@{--}[rr] \ar[ur] && 23 \ar@{--}[rr] \ar[ur] && 34 \ar@{--}[rr] \ar[ur]  && 45 \ar@{--}[rr] \ar[ur] && 51 \ar@{--}[rr] \ar[ur] && 12 \ar[ur] \\
}$$
\caption{The translation quiver $\Gamma(5)$.}
\label{f:gamma4}
\end{figure}
\end{example}

By~\cite[2.3]{pengxiao} (see also~\cite[9.9]{keller}), $\mathcal{R}_n$ is a triangulated category, and,
by~\cite[1.3]{bmrrt}, it has Auslander-Reiten triangles and its Auslander-Reiten
quiver, $\Gamma(\mathcal{R}_n)$ is the quotient of the Auslander-Reiten quiver
of $D^b(kQ)$ by the automorphism induced by $[2]$.
We have the following:

\begin{prop} \label{p:rootcatgamma}
The translation quiver $\Gamma(n)$ is isomorphic to $\Gamma(\mathcal{R}_n)$.
\end{prop}

\begin{proof}
Suppose that $Q$ is a linearly oriented quiver of type $A_{n-1}$, with arrows
$i\leftarrow i+1$, $1\leq i\leq n-2$. Then, up to isomorphism, the
indecomposable modules
for $kQ$ are of the form $M_{ij}$, $1\leq i<j\leq n$, where
$M_{ij}$ has socle $S_i$ (the simple module corresponding to vertex $i$) and
length $j-i$. So $M_{i,i+1}=S_i$. For $i<j$, we also set $M_{ji}=M_{ij}[1]$.
Then the map $[i,j]\mapsto M_{ij}$, for $1\leq i,j\leq n$, $i\neq j$,
gives a bijection between oriented edges between vertices of $\PP_n$ and
isomorphism
classes of indecomposable objects of $\mathcal{R}_n$. The fact that this is an
isomorphism of translation quivers follows from the description of the
Auslander-Reiten quiver of $D^b(kQ)$ in~\cite{happel2}.
We just need to check that the mesh beginning at corresponding vertices is
the same in each quiver.
The only non-trivial cases are the meshes beginnning with $M_{ij}$ where
$j=n$ or $i=n$. In the first case, the mesh in $\Gamma(\mathcal{R}_n)$ is:
$$\xymatrix{
&M_{i,1}=P_{i-1}[1] \ar[dr] & \\
M_{i,n}=I_i \ar[ur] \ar[dr] \ar@{--}[rr] && M_{i+1,1}[1]=P_i[1] \\
& M_{i+1,n}=I_{i+1} \ar[ur] &
},$$
and in the second case, the mesh in $\Gamma(\mathcal{R}_n)$ is:
$$\xymatrix{
&M_{n,i+1}=I_{i+1}[1] \ar[dr] & \\
M_{n,i}=I_i[1] \ar[ur] \ar[dr] \ar@{--}[rr] && M_{1,i+1}=P_i \\
& M_{1,i}=P_{i-1} \ar[ur] &
},$$
noting that in $\mathcal{R}_n$, $X\cong X[2]$ for any object $X$.
These meshes are the images of the corresponding meshes in $\Gamma(n)$,
so we are done.
\end{proof}

\begin{remark}
We note that the induced subquiver of $\Gamma(n)$ on vertices of form $[i,j]$
with $i<j$ (with $\tau([i,j])$ undefined if $i=1$) is isomorphic to the
Auslander-Reiten quiver of $kQ-\module$.
\end{remark}

We note that, as for the cluster category (see~\cite[\S1]{bmrrt}), the
category $\mathcal{R}_n$ is standard. We thus have the following corollary
of Proposition~\ref{p:rootcatgamma}, giving a geometric realisation of
$\mathcal{R}_n$.

\begin{corollary}
The root category $\mathcal{R}_n$ is equivalent to the additive hull of the mesh category of
$\Gamma(n)$.
\end{corollary}

We shall identify indecomposable objects in $\mathcal{R}_n$, up to
isomorphism, with the
corresponding oriented edges between vertices of $\PP_n$, in the sequel,
and we shall freely switch between objects and oriented edges between vertices in
$\PP_n$.

\section{Dimensions of extension groups}

In this section we indicate how the dimensions of $\Ext^1$-groups between
indecomposable objects of $\mathcal{R}_n$ can be read off from the geometric
model. We fix $i,j$ with $1\leq i,j\leq n$, $i\neq j$,
and consider the corresponding indecomposable object $[i,j]$.

Consider the two rectangles
$R_B=R_B(i,j)$, with corners $[j,i]$, $[j,j-1]$, $[i-1,j-1]$ and $[i-1,i]$,
and $R_F=R_F(i,j)$, with corners
$[i+1,j+1]$, $[i+1,i]$, $[j,i]$ and $[j,j+1]$, in $\mathcal{R}_n$.

Note that the Auslander-Reiten formula holds in $D^b(kQ)$.
This states that, for two indecomposable objects $X,Y$ in $D^b(kQ)$,
$$\Ext^1_{D^b(kQ)}(X,Y)\cong D\Hom_{D^b(kQ)}(Y,\tau X),$$
where $D=\Hom_k(-,k)$. This passes down to the root category $\mathcal{R}_n$
and it follows that the dimensions of $\Ext^1_{\mathcal{R}_n}(X,Y)$ and
$\Hom_{\mathcal{R}_n}(Y,\tau X)$ coincide for any two indecomposable
objects $X,Y$ in $\mathcal{R}_n$.

Using~\cite{bongartz} or the mesh relations and the Auslander-Reiten
formula in $\mathcal{R}_n$ directly, we see that:

\begin{lemma} \label{l:originalextrule}
Let $X$ and $Y$ be indecomposable objects in $\mathcal{R}_n$. Then:
\begin{enumerate}
\item[(a)]
The space $\Ext^1_{\mathcal{R}_n}([i,j],Y)$ is non-zero if and only if $Y$ lies in $R_B$.
If it is non-zero then it is one-dimensional.
\item[(b)]
The space $\Ext^1_{\mathcal{R}_n}(X,[i,j])$ is non-zero if and only if $X$ lies in $R_F$.
If it is non-zero then it is one-dimensional.
\end{enumerate}
\end{lemma}

As an example, we show in Figure~\ref{f:extexample} (by underlining)
those indecomposable objects $Y$ such that
$\Ext^1_{\mathcal{R}_5}([3,1],Y)\neq 0$ for the case $n=5$, and
(by overlining) those indecomposable objects $X$ such that
$\Ext^1_{\mathcal{R}_{5}}(X,[3,1])\neq 0$.
In each case the objects in question
are written in bold font.
Note that $[1,3]$ is the only object satisfying both conditions.

Let $X,Y$ be two indecomposable objects of $\mathcal{R}_n$, regarded
as oriented edges between vertices of $\PP_n$. Then $X$ and $Y$ may cross each other
in two different ways. If the tangents to $X,Y$ (in that order) form a pair
of axes corresponding to the usual orientation on $\mathbb{R}^2$ we say that
the crossing of $X$ and $Y$ is positive, otherwise negative. See
Figure~\ref{f:crossingtypes}.

\begin{figure}
\psfragscanon
\psfrag{X}{$X$}
\psfrag{Y}{$Y$}
\subfigure[Positive crossing]{\includegraphics[width=3cm]{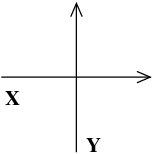}}
\hskip 1cm
\subfigure[Negative crossing]{\includegraphics[width=3cm]{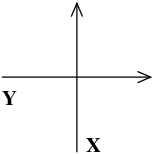}}
\caption{The two types of crossing between $X$ and $Y$.}
\label{f:crossingtypes}
\end{figure}

Lemma~\ref{l:originalextrule} can be reinterpreted geometrically as follows.

\begin{prop} \label{p:extrule}
Let $X,X'$ be indecomposable objects in $\mathcal{R}_n$.
Then $\dim \Ext^1_{\mathcal{R}_n}(X,X')$ is equal to $1$ if and only if one of
the following conditions holds, and is zero otherwise:
\begin{enumerate}
\item[(a)] The crossing of $X,X'$ is positive; 
\item[(b)] The terminal vertex of $X$ coincides with the initial vertex of $X'$ and
$X'$ lies to the left of $X$ in $\PP_n$;
\item[(c)] The initial vertex of $X$ coincides with the terminal vertex of $X'$ and
$X'$ lies to the right of $X$ in $\PP_n$;
\item[(d)] $X'$ is the reverse of $X$.
\end{enumerate}
\end{prop}

As an example, consider the objects $Y$ such that
$\Ext^1_{\mathcal{R}_5}([3,1],Y)\not=0$, displayed
in Figure~\ref{f:extexample} (by underlining).
Note that the crossing of $[3,1]$ with $[2,5]$ or with $[2,4]$ is positive; 
both $[1,4]$ and $[1,5]$ start at the terminal vertex of $[3,1]$ and
lie to its left; $[2,3]$ has terminal vertex coinciding with the initial
vertex of $[3,1]$ and lies to its right, and $[1,3]$ is the reverse of $[3,1]$.

\begin{figure}
$$\xymatrix@C=0.25cm{
&&& \underline{15} \ar@{--}[rr] \ar[dr] && 21 \ar@{--}[rr] \ar[dr] && 32 \ar@{--}[rr] \ar[dr] && \overline{43} \ar@{--}[rr] \ar[dr] && 54 \ar@{--}[rr] \ar[dr] && \underline{15} \ar@{--}[rr] \ar[dr] && 21\\
&& \underline{14} \ar@{--}[rr] \ar[ur] \ar[dr] && \underline{25} \ar@{--}[rr] \ar[ur] \ar[dr] && 31 \ar@{--}[rr] \ar[ur] \ar[dr] && \overline{42} \ar@{--}[rr] \ar[dr] \ar[ur] && \overline{53} \ar@{--}[rr] \ar[ur] \ar[dr]&& \underline{14} \ar[ur] \ar@{--}[rr] \ar[dr] && \underline{25} \ar[ur] \\
& \underline{\overline{13}} \ar@{--}[rr] \ar[ur] \ar[dr] && \underline{24} \ar@{--}[rr] \ar[ur] \ar[dr] && 35 \ar@{--}[rr] \ar[ur] \ar[dr] && 41 \ar@{--}[rr] \ar[ur] \ar[dr] && \overline{52} \ar@{--}[rr] \ar[ur] \ar[dr] && \underline{\overline{13}} \ar[ur] \ar@{--}[rr] \ar[dr] && \underline{24} \ar[ur] \\
\overline{12} \ar@{--}[rr] \ar[ur] && \underline{23} \ar@{--}[rr] \ar[ur] && 34 \ar@{--}[rr] \ar[ur]  && 45 \ar@{--}[rr] \ar[ur] && 51 \ar@{--}[rr] \ar[ur] && \overline{12} \ar[ur] \ar@{--}[rr] && \underline{23} \ar[ur] \\
}$$
\caption{Objects with non-trivial extensions with $[3,1]$ in either direction, in $\mathcal{R}_5$.}
\label{f:extexample}
\end{figure}
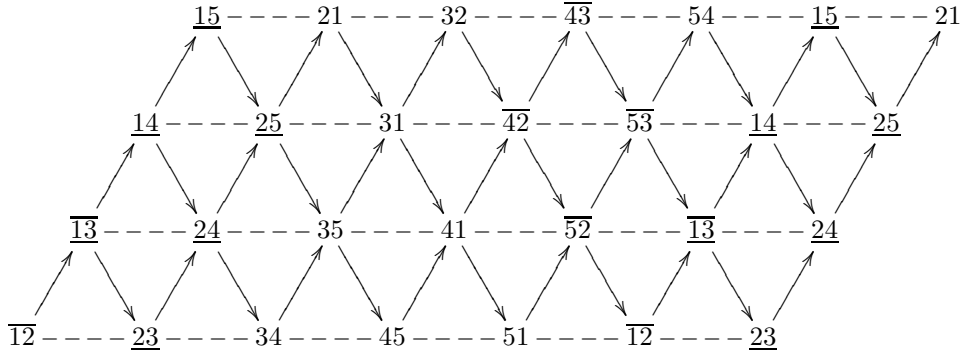

\section{Starting and ending frames} \label{s:frames}

In this section we consider starting and ending frames of indecomposable
objects in $\mathcal{R}_n$ (following~\cite[8.4]{bmrrt}). Let $\ind(\mathcal{R}_n)$
denote the set of (isomorphism classes of) indecomposable objects of
$\mathcal{R}_n$.
Let $X$ be an indecomposable object in $\mathcal{R}_n$. Then the
\emph{starting frame} $S(X)$ of $X$ is the set
$$S(X)=\{Y\in \ind(\mathcal{R}_n)\,:\,\Hom_{\mathcal{R}_n}(X,Y)\neq 0,\,\Ext^1_{\mathcal{R}_n}(Y,X)=0\}.$$
The \emph{ending frame} $E(X)$ of $X$ is the set
$$E(X)=\{Y\in \ind(\mathcal{R}_n)\,:\,\Hom_{\mathcal{R}_n}(Y,X)\neq 0,\,\Ext^1_{\mathcal{R}_n}(X,Y)=0\}.$$
We define the \emph{frame} $F(X)$ of $X$ to be the union:
$$F(X)=S(X)\cup E(X).$$

The following can be checked by direct calculation or using
Proposition~\ref{p:extrule} and the Auslander-Reiten formula.

We also note that, if $\overline{X}$ denotes the reverse
of $X$, then
$$\Hom_{\mathcal{R}_n}(\overline{X},Y) \cong \Hom_{\mathcal{R}_n}(X[1],Y)
\cong \Hom_{\mathcal{R}_n}(X,Y[1]) \cong \Ext_{\mathcal{R}_n}(X,Y)$$
and
$$\Ext_{\mathcal{R}_n}(\overline{X},Y) \cong \Ext_{\mathcal{R}_n}(X[1],Y)
\cong \Ext_{\mathcal{R}_n}(X,Y[-1]) \cong \Hom_{\mathcal{R}_n}(X,Y).$$

\begin{prop} \label{p:framerule}
\begin{enumerate}
\item[(a)] $Y\in S(X)$ if and only if $Y$ and $X$ share a common terminal vertex
and $Y$ lies to the left of $X$, or $Y$ and $X$ share a common initial vertex
and $Y$ lies to the right of $X$, or $Y=X$.
\item[(b)] $Y\in E(X)$ if and only if $Y$ and $X$ share a common initial vertex
and $Y$ lies to the left of $X$, or $Y$ and $X$ share a common terminal vertex
and $Y$ lies to the right of $X$, or $Y=X$.
\item[(c)] $Y\in F(X)$ if and only if $X$ and $Y$ share a common initial
vertex or share a common terminal vertex (or both).
\end{enumerate}
\end{prop}

\section{Categorification of the determinantal result}
Fix again a triangulation $\pi$ of $\PP_n$.
For $1\leq i,j\leq n$ with $i\neq j$,
we associate the indecomposable object $[i,j]$
(or oriented edge) with the $i,j$ position in the $n \times n$ matrix
$U(\pi)$ considered in Section~\ref{s:introduction}.
The indecomposable $kQ$-modules
(with $i<j$) correspond to the part of $U(\pi)$ above the leading diagonal
and their shifts (with $i>j$) correspond to the part of $U(\pi)$ below the
leading diagonal.

Reinterpreting Proposition~\ref{p:framerule} in these terms, we obtain:

\begin{lemma} \label{l:matrixframerule}
The frame of an indecomposable object $X$ corresponds to the union of the
row and column of $U(\pi)$ containing $X$ (apart from the diagonal entries).
\end{lemma}

Define a \emph{frame-free} configuration of $\mathcal{R}_n$ to be a
maximal collection $C$ of (isomorphism classes of) indecomposable objects of
$\ind\mathcal{R}_n$ such that $Y\not\in F(X)$ for all $X,Y$ in $C$.
Thus frame-free configurations of $\mathcal{R}_n$ correspond to maximal
collections of positions in $U(\pi)$ which do not lie in the same row or column
as each other and contain no diagonal entries.

\begin{lemma} \label{l:framecardinality}
Let $C$ be a frame-free configuration of $\mathcal{R}_n$. Then the
cardinality of $C$ is either $n-1$ or $n$.
\end{lemma}

\begin{proof}
Since frame-free configurations cannot have objects in the same row or
column, the maximum cardinality is $n$. If the cardinality of a configuration
is $n-k$ where $k\geq 2$, there are at least two rows and two columns of
$U(\pi)$ containing no elements of the configuration, a contradiction to
its maximality (as at least two elements could be added to the configuration,
at the non-diagonal intersections of the empty rows and columns). The result
follows.
\end{proof}

Given a fixed-point free permutation, $\sigma$
(sometimes known as a \emph{derangement}),
let $C(\sigma)$ be the set of objects $[i,\sigma(i)]$
for $1\leq i\leq n$. Since $\sigma$ is fixed-point free, it follows from
Proposition~\ref{p:framerule}(c) that $C(\sigma)$ is a frame-free
configuration; it has cardinality $n$. It is clear that this gives a bijection
between fixed-point free permutations and frame-free configurations of
cardinality $n$.

Thus, as a collection of oriented edges between vertices of $\PP_n$, a frame-free
configuration in $\mathcal{R}_n$ of cardinality $n$ is a union of oriented cycles
(with no cycles of cardinality $1$).

Similarly, a permutation $\sigma$ with one fixed point, $p$, say, corresponds
to a configuration of cardinality $n-1$ consisting of the objects
$[i,\sigma(i)]$ for $i\neq p$. This gives a bijection between the
permutations with a single fixed point and the frame-free configurations
of cardinality $n-1$.

The number of permutations with a given number of fixed points is well-known
(see e.g.~\cite[A008290]{sloane}). We thus have the following:

\begin{lemma}
The number of frame-free configurations of $\mathcal{R}_n$ of cardinality
$n$ is $$!n:=n!\sum_{k=0}^n \frac{(-1)^k}{k!},$$
(known as the subfactorial of $n$) while the number of frame-free
configurations of $\mathcal{R}_n$ of cardinality $n-1$ is $!n+(-1)^{n-1}$.
\end{lemma}

Thus the number of frame-free configurations in $\mathcal{R}_n$ of
cardinality $n$ for $n=2,3,\ldots$, is $1,2,9,44,265,1854$ (see~\cite[A000166]{sloane}) and the number of frame-free
configurations in $\mathcal{R}_n$ of cardinality $n-1$ is $0,3,8,45,264,1855$ (see~\cite[A000240]{sloane}).
The $9$ frame-free configurations of $\mathcal{R}_4$ of cardinality $4$
are shown in Figure~\ref{f:exampleconfigs} as collections of vertices in the AR-quiver
(filled in vertices indicate those indecomposable objects in the collection) and as
collections of oriented arcs between vertices of a square.

\begin{figure}
\psfragscanon
\psfrag{1}{$1$}
\psfrag{2}{$2$}
\psfrag{3}{$3$}
\psfrag{4}{$4$}
\psfrag{m1}{$-u_{12}u_{23}u_{34}u_{41}$}
\psfrag{m2}{$-u_{12}u_{24}u_{31}u_{43}$}
\psfrag{m3}{$-u_{13}u_{24}u_{32}u_{41}$}
\psfrag{m4}{$-u_{13}u_{21}u_{34}u_{42}$}
\psfrag{m5}{$-u_{14}u_{23}u_{31}u_{42}$}
\psfrag{m6}{$-u_{14}u_{21}u_{32}u_{43}$}
\psfrag{m7}{$+u_{12}u_{21}u_{34}u_{43}$}
\psfrag{m8}{$+u_{13}u_{24}u_{31}u_{42}$}
\psfrag{m9}{$+u_{14}u_{23}u_{32}u_{41}$}
\includegraphics[width=12cm]{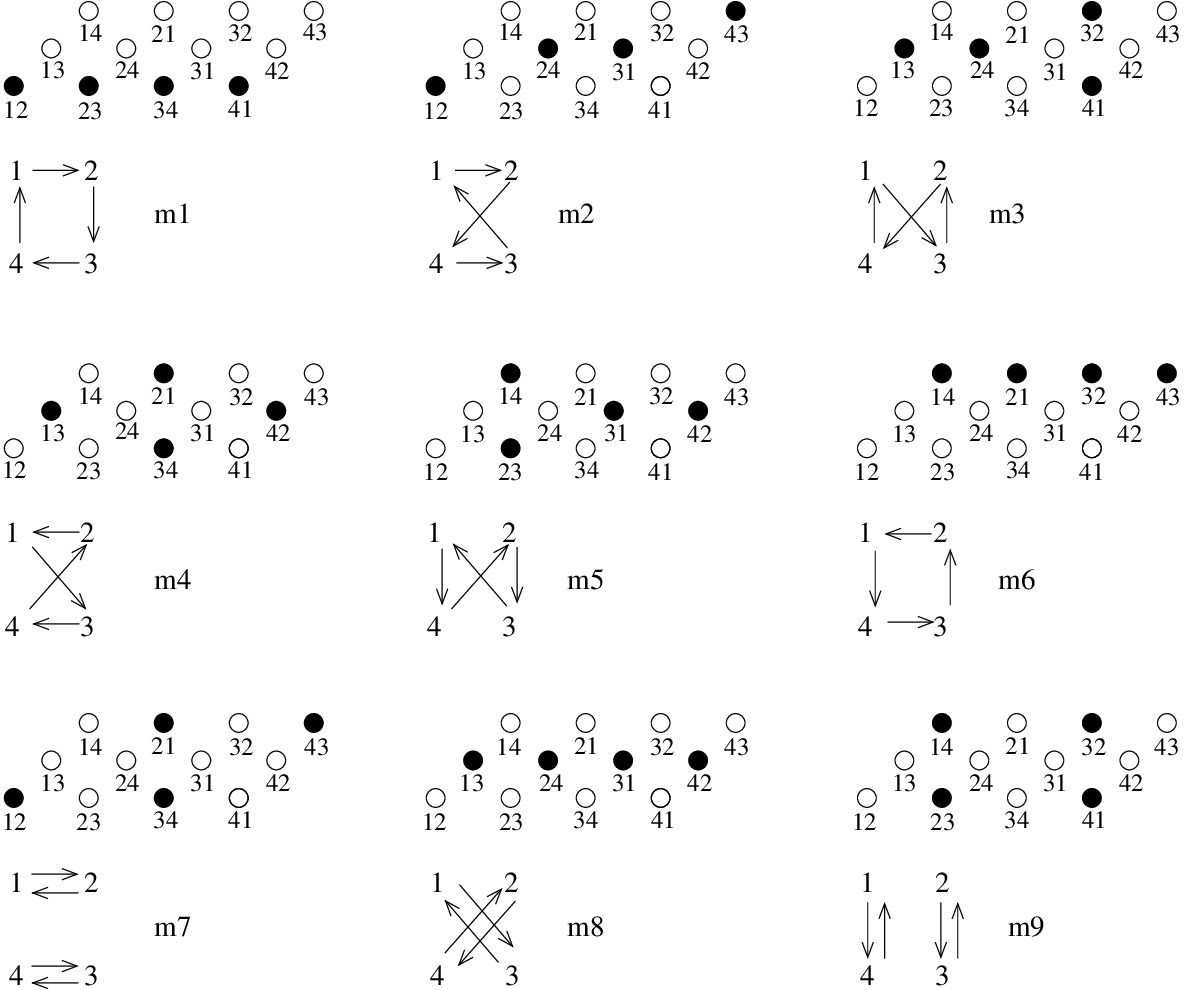}
\caption{The $9$ frame-free configurations of cardinality $4$ in the root category of
type $A_3$ and the corresponding terms in the expansion of $\det(U(\pi))$.}
\label{f:exampleconfigs}
\end{figure}

We see that the frame-free configurations of $\mathcal{R}_n$ of cardinality
$n$ correspond bijectively to the non-zero terms in the expansion
$$\det(U(\pi))=\sum_{\sigma\in\Sigma_{n}} (-1)^{\ell(\sigma)} u_{1,\sigma(1)}\cdots u_{n,\sigma(n)}$$
of the determinant of $U(\pi)$ (since $U(\pi)$ has zeros along its leading diagonal).

Given a frame-free configuration $C$ of $\mathcal{R}_n$ of
cardinality $n$, define its \emph{sign} $\varepsilon(C)$ to be
$$\varepsilon(C)=\prod_{\gamma} (-1)^{\ell(\gamma)-1}$$
where the product is over the oriented cycles in the representation of
$C$ as a collection of oriented edges between vertices of $\PP_n$, and
$\ell(\gamma)$ is equal to the number
of vertices in $\gamma$ for a cycle $\gamma$. It is easy to see that this
is equal to the sign of the corresponding permutation. Set $\alpha(C)$ to
be the product of the entries of $U(\pi)$ corresponding to the elements of $C$.

We therefore have:
$$\det(U(\pi))=\sum_{C} \varepsilon(C)\alpha(C),$$
where the sum is over all frame-free configurations of $\mathcal{R}_n$
of cardinality $n$. The monomials in this expansion are shown for each
frame-free configuration in the example in Figure~\ref{f:exampleconfigs}.

We can reinterpret Theorem~\ref{t:clusterinterpretation}
representation-theoretically as follows. Recall that $\mathcal{A}$
is a cluster algebra of type $A_{n-3}$ with coefficients corresponding to
the boundary edges of $\mathcal{P}_n$.

\begin{theorem} \label{t:detresult2}
Let $n\geq 3$, let $\pi$ be a triangulation of $\PP_n$, and
let $U(\pi)=(u_{ij})$ be the matrix of cluster variables in $\mathcal{A}$
regarded as Laurent polynomials in the $u_{ij}$ for $i,j$ end-points of
diagonals in $\pi$ with coefficients given by polynomials in the stable
variables. Then we have:
$$\sum_{C} \varepsilon(C)\alpha(C)=-(-2)^{n-2}u_{12}u_{23}\cdots u_{n-1,n}u_{n,1},$$
where the sum is over all frame-free configurations of
$\mathcal{R}_n=D^b(kQ)/[2]$ of maximum cardinality.
\end{theorem}

\noindent \textbf{Acknowledgements:} We would like to thank the referees for
their helpful and interesting comments. RJM would like to thank Karin Baur
and the FIM at the ETH, Z\"{u}rich for their support and kind hospitality.

\end{document}